\documentclass[a4paper]{amsart}

\usepackage[latin1]{inputenc}
\usepackage{amsfonts}
\usepackage{amssymb}
\usepackage[leqno]{amsmath}
\usepackage{amsthm}
\usepackage{amscd}
\usepackage[all]{xypic}
\usepackage{enumerate}
\usepackage[mathscr]{eucal}
\usepackage{ulem,graphicx}            
\usepackage{verbatim}

\usepackage{times}

\newtheorem{theorem}{Theorem}[section]

\newtheorem{corollary}[theorem]{Corollary}

\theoremstyle{definition}

\newtheorem{remark}[theorem]{Remark}

\DeclareMathOperator{\ev}{ev}
\DeclareMathOperator{\id}{id}
\DeclareMathOperator{\Hom}{Hom}

\DeclareMathOperator{\Spec}{Spec}
\DeclareMathOperator{\spec}{spec}   
\DeclareMathOperator{\rk}{rk}       
\DeclareMathOperator{\Gr}{Gr}      
      
\DeclareMathOperator{\GL}{GL}      
\DeclareMathOperator{\Quot}{Quot}

\DeclareMathOperator{\Ob}{Ob}

\def\op{\textup{op}}

\def\N{\mathbb{N}}
\def\Z{\mathbb{Z}}
\def\R{\mathbb{R}}
\def\Q{\mathbb{Q}}
\def\C{\mathbb{C}}

\def\O{\mathcal{O}}

\def\ot{\otimes}
\def\vphi{\varphi}
\def\F{\mathbb{F}}
\def\lto{\longrightarrow}

\def\lmto{\longmapsto}
\def\Fun{{\F_1}}                             
\def\Gm{\mathbb{G}_m}

\def\bdu{\coprod} 

\def\Sets{\textrm{$\mathcal{S}$ets}}

\def\Sh{\textrm{$\mathcal{S}$h}}

\def\Psh{\textrm{$\mathcal{P}$sh}}

\DeclareMathOperator{\Frob}{Frob}
\def\op{\textup{op}}
\def\rk{\textup{rk}}
\def\cA{\mathcal A}
\def\cB{\mathcal B}

\def\cD{\mathcal D}
\def\cF{\mathcal F}
\def\cG{\mathcal G}
\def\cM{\mathcal M}
\def\cO{\mathcal O}
\def\cT{\mathcal T}
\def\cX{\mathcal X}

\def\p{\mathfrak{p}}
\def\q{\mathfrak{q}}

\newcommand{\dtext}[1]{\emph{#1}} 

\newcommand{\MMod}[1]{#1\textrm{-Mod}}  

\title[An overview of geometries over the field with one element]{Mapping $\Fun$-land:\\ An overview of geometries\\ over the field with one element}
\author{Javier L\'{o}pez Pe\~na}
\address{Mathematics Research Centre\\ 
Queen Mary University of London\\
Mile End Road, London E1 4NS, United Kingdom}
\email{jlopez@maths.qmul.ac.uk \textrm{(J. L\'opez Pe\~na)}}
\thanks{J. L\'opez Pe\~na was supported by the EU Marie-Curie fellowship PIEF-GA-2008-221519 at Queen Mary University of London.}
\author{Oliver Lorscheid}
\address{Max-Planck Institut f\"ur Mathematik\\
Vivatsga\ss{}e, 7. D-53111, Bonn, Germany}
\email{oliver@mpim-bonn.mpg.de \textrm{(O. Lorscheid)}}
\thanks{O. Lorscheid was supported by the Max-Planck-Institut f\"ur Mathematik in Bonn.}

\begin{document}

\begin{abstract}
 This paper gives an overview of the various approaches towards $\Fun$-geometry. In a first part, we review all known theories in literature so far, which are: Deitmar's $\Fun$-schemes, To\"en and Vaqui\'e's $\Fun$-schemes, Haran's $\F$-schemes, Durov's generalized schemes, Soul\'e's varieties over $\Fun$ as well as his and Connes-Consani's variations of this theory, Connes and Consani's $\Fun$-schemes, the author's torified varieties and Borger's $\Lambda$-schemes. In a second part, we will tie up these different theories by describing functors between the different $\Fun$-geometries, which partly rely on the work of others, partly describe work in progress and partly gain new insights in the field. This leads to a commutative diagram of $\Fun$-geometries and functors between them that connects all the reviewed theories. We conclude the paper by reviewing the second author's constructions that lead to realization of Tits' idea about Chevalley groups over $\Fun$.
 
\end{abstract}

\maketitle

\tableofcontents


\section*{Introduction}

The hour of birth to the field with one element was given in Jacques Tits' paper \cite{Tits1957} from 1956, in which he indicated that the analogy between the symmetric group $S_n$ and the Chevalley group $\GL_n(\F_q)$ (as observed by Robert Steinberg in \cite{Steinberg1951} in 1951) should find an explanation by interpreting $S_n$ as a Chevalley group over the ``field of characteristic one''. Though Tits' idea was not taken serious at that time, it is one of the guiding thoughts in the development of $\Fun$-geometry today. For a recent treatment of Steinberg's paper, see \cite{Soule2009}, and for an overview of what geometric objects over $\F_q$ should become if $q=1$, see \cite{Cohn2004}.

It took more than 35 years that the field of one element reoccurred in mathematical literature. In an unpublished note (\cite{KapranovUN}), Mikhail Kapranov and Alexander Smirnov developed the philosophy that the set of the $n$-th roots should be interpreted as $\F_{1^n}$, a field extension of $\Fun$ in analogy to the field extension $\F_{p^n}$ of $\F_p$. A scheme that contains the $n$-th roots of unity should be thought of as a scheme over $\F_{1^n}$. The tensor product $\F_{1^n}\otimes_\Fun\Z$ should be the group ring $\Z[\Z/n\Z]$.

In the early nineties, Christoph Deninger published his studies (\cite{Deninger1991}, \cite{Deninger1992}, \cite{Deninger1994}, et cetera) on motives and regularized determinants. In his paper \cite{Deninger1992}, Deninger gave a description of conditions on a category of motives that would admit a translation of Weil's proof of the Riemann hypothesis for function fields of projective curves over finite fields $\F_q$ to the hypothetical curve $\overline{\Spec\Z}$. In particular, he showed that the following formula would hold:
	\begin{multline*} 
		2^{-1/2}\pi^{-s/2}\Gamma(\frac s2)\zeta(s) \ = \\ 
		\frac{\det_\infty\Bigl(\frac 1{2\pi}(s-\Theta)\Bigl| H^1(\overline{\Spec\Z},\cO_\cT)\Bigr.\Bigr)}{\det_\infty\Bigl(\frac 1{2\pi}(s-\Theta)\Bigl| H^0(\overline{\Spec\Z},\cO_\cT)\Bigr.\Bigr)\det_\infty\Bigl(\frac 1{2\pi}(s-\Theta)\Bigl| H^2(\overline{\Spec\Z},\cO_\cT)\Bigr.\Bigr)} 
	\end{multline*}
where $\det_\infty$ denotes the regularized determinant, $\Theta$ is an endofunctor that comes with the category of motives and $H^i(\overline{\Spec\Z},\cO_\cT)$ are certain proposed cohomology groups. This description combines with Nobushige Kurokawa's work on multiple zeta functions (\cite{Kurokawa1992}) from 1992 to  the hope that there are motives $h^0$ (\emph{``the absolute point''}), $h^1$ and $h^2$ (\emph{``the absolute Tate motive''}) with zeta functions
	\[ 
		\zeta_{h^w}(s) \ = \ \det{}_\infty\Bigl(\frac 1{2\pi}(s-\Theta)\Bigl| H^w(\overline{\Spec\Z},\cO_\cT)\Bigr.\Bigr) 
	\]
for $w=0,1,2$. Deninger computed that $\zeta_{h^0}(s)=s/2\pi$ and that $\zeta_{h^2}(s)=(s-1)/2\pi$. It was Yuri Manin who proposed in \cite{Manin1995} the interpretation of $h^0$ as $\Spec\Fun$ and the interpretation of $h^2$ as the affine line over $\Fun$. The quest for a proof of the Riemann hypothesis was from now on a main motivation to look for a geometry over $\Fun$. Kurokawa continued his work on zeta functions in the spirit of $\Fun$-geometry in the collaboration \cite{Kurokawa2002} with Hiroyuki Ochiai and Masato Wakayama and in \cite{Kurokawa2005}. 

In 2004, Christophe Soul\'e proposed a first definition of an algebraic variety over $\Fun$  in \cite{Soule2004}, based on the observation that the extension of the base field of a scheme can be characterized by a universal property. His suggestion for a variety over $\Fun$ is an object involving a functor, a complex algebra, a scheme and certain morphisms and natural transformations such that a corresponding universal property is satisfied. Shortly after that, many different approaches to $\Fun$-geometry arose.

Anton Deitmar reinterpreted in \cite{Deitmar2005} the notion of a fan as given by Kazuya Kato in \cite{Kato1994} as a scheme over $\Fun$. He calculated zeta functions for his $\Fun$-schemes (\cite{Deitmar2006}) and showed that the $\Fun$-schemes whose base extension to $\C$ are complex varieties correspond to toric varieties (\cite{Deitmar2007}). Bertrand T\"oen and Michel Vaqui\'e associated to any symmetric monoidal category with certain additional properties a category of geometric objects. In the case of the category of sets together the cartesian product, the geometric objects are locally representable functors on monoids, which are $\Fun$-schemes in the sense of To\"en and Vaqui\'e. Florian Marty developed theory on Zariski open objects (\cite{Marty2007}) and smooth morphisms (\cite{Marty2009}) in this context.

Shai Haran (cf.\ \cite{Haran2007}) proposed using certain categories modeled over finite sets as a replacement for rings, and actually produced a candidate for the compactification $\overline{\Spec \Z}$ of $\Spec\Z$ in his framework. Nikolai Durov developed in \cite{Durov2007} an extension of classical algebraic geometry within a categorical framework that essentially implied replacing rings by a certain type of monads. As a byproduct of his theory he obtained a definition of $\Fun$ and an algebraic geometry attached to it. See \cite{Fresan2009} for a summary.

In 2008, Alain Connes, Katia Consani and Matilde Marcolli showed (\cite{Connes2008a}) that the Bost-Connes system defined in \cite{Bost1995} admits a realization as a geometric object in the sense of Soul\'e. The suggestion of Soul\'e to consider a variation of the functor in his approach and other ideas, led to Connes and Consani the variation of Soul\'e's approach as presented in \cite{Connes2008}. That paper contains a first contribution to Tits' idea of Chevalley groups over $\Fun$, namely, Connes and Consani define them as varieties over $\F_{1^2}$. Soul\'e himself wrote only later a text with his originally suggested modification. It can be found in this volume (\cite{Soule2009}).

Yuri Manin proposed in \cite{Manin2008} a notion of analytic geometry over $\Fun$. The key idea is that one should look for varieties having ``enough cyclotomic points'', an thought that relates to Kazuo Habiro's notion (\cite{Habiro2004}) of the cyclotomic completion of a polynomial ring, which finds the interpretation as the ring of analytic functions on the set of all roots of unity. Matilde Marcolli presents in \cite{Marcolli2009} an alternative model to the BC-system for the noncommutative geometry of the cyclotomic tower. For that purpose, she uses the multidimensional analogues to the Habiro ring defined by Manin, and constructs a class of multidimensional BC-endomotives. Marcolli's endomotives turn out to be closely related to $\Lambda$-rings, in the sense of \cite{Borger2008}.

Aiming at unifying the different notions of varieties over $\Fun$ (after Soul\'e and Connes-Consani) as well as establishing new examples of $\Fun$-varieties, the authors of this text introduced in \cite{LL} the notion of torified varieties. Of particular interest were Grassmann varieties, which are shown to be torified varieties and to provide candidates of $\Fun$-varieties. However, these candidates fail to satisfy the constraints of Soul\'e's and Connes-Consani $\Fun$-geometries. Independently of this work, Connes and Consani introduced in \cite{Connes2009} a new notion of scheme over $\Fun$, which simplified the previous approaches by Soul\'e and themselves by merging Deitmar's and To\"en-Vaquies viewpoints into it. We will show in this paper that this notion is closely related to torified varieties. The second author showed in \cite{Lorscheid2009} that Connes-Consani's new notion of $\Fun$-geometry is indeed suitable to realize Tits' original ideas on Chevalley groups over $\F_1$.

Another promising notion of $\Fun$-geometry was given recently by James Borger in \cite{Borger2009}, who advocates the use of $\Lambda$-ring structures as the natural descent data from $\Z$ to $\Fun$.

The aim of this paper is to give an overview of the new land of geometries over the field with one element. Firstly, we review briefly the different developments and building bricks of $\Fun$-geometries. Secondly, we tie up and lay paths and bridges between the different $\Fun$-geometries by describing functors between them. This will finally lead to a commutative diagram of $\Fun$-geometries and functors between them that can be considered as a first map of $\Fun$-land.

The paper is organized as follows. In part \ref{bricks}, we review the different notions of $\Fun$-geometries. We describe a theory in detail where it seems important to gain an insight into its technical nature, but we will refrain from a detailed treatment in favor of a rough impression where technicalities lead too far for a brief account. In these cases, we provide the reader with references to the literature where the missing details can be found.

In the order of the paper, the following approaches towards $\Fun$-geometries are reviewed: Deitmar's $\Fun$-schemes (which we call $\cM$-schemes in the following) are described in section \ref{mschemes}, To\"en and Vaqui\'e's $\Fun$-schemes in section \ref{tvschemes}, Haran's $\F$-schemes in section \ref{hschemes} and Durov's generalized schemes in section \ref{durov}. In section \ref{svarietiesandco}, we give a taste of Soul\'e's definition of a variety over $\Fun$ and describe his and Connes-Consani's variations on this notion. In section \ref{ccschemes}, we present Connes and Consani's new proposal of $\Fun$-geometry including the notion of an $\cM_0$-scheme. In section \ref{torifsch}, we review the authors' definition of a torified variety. Finally, we give in section \ref{lambdaschemes} an insight into Borger's $\Lambda$-schemes.

In part \ref{bridges}, we review and construct functors between the categories introduced in the first part of the paper. The very central objects of $\Fun$-geometries are toric varieties. As we will see, these can be realized in all considered $\Fun$-geometries. We begin in section \ref{torictom} with recalling the definition of a toric variety and following Kato (\cite{Kato1994}) to show that toric varieties are $\Fun$-schemes after Deitmar, i.e.\ $\cM$-schemes. In section \ref{mtotv} we describe Florian Marty's work on comparing Deitmar's and To\"en-Vaqui\'e's notions of $\Fun$-schemes. 

In section \ref{mtocc}, we lay the path from $\cM$-schemes to $\cM_0$-schemes and Connes-Consani's $\Fun$-schemes. In section \ref{mtoh}, we recall from Haran's paper \cite{Haran2007} that $\cM$-schemes and $\cM_0$-schemes are $\F$-schemes. In section \ref{mtodurov}, we see that the same result holds true for Durov's generalized schemes (with $0$) in place of $\F$-schemes. In section \ref{htodurov}, we review Peter Arndt's work in progress about comparing Haran's and Durov's approaches towards $\Fun$-geometry. In section \ref{mtolambda}, we refer to Borger's paper \cite{Borger2009} to establish $\cM$-schemes as $\Lambda$-schemes. 

In section \ref{torictotorified}, we recall from the authors' previous paper \cite{LL} that toric varieties are affinely torified. In section \ref{torifiedtos}, we give an idea of why affinely torified varieties define varieties over $\Fun$ (after Soul\'e resp.\ Connes and Consani). In section \ref{torifiedtocc}, we extend the definition of a torified variety to a generalized torified scheme in order to show that the idea behind the notion of a torified variety and an $\Fun$-scheme after Connes and Consani are the same. All these categories and functors between them will be summarized in the diagram in Figure \ref{bigpicture} of section \ref{map}. 

We conclude the paper with a review of the realization of Tits' ideas on Chevalley groups over $\F_1$ by the second author in section \ref{algebraic_groups}.

{\bf Acknowledgments:} The authors thank the organizers and participants of the Nashville conference on $\Fun$-geometry for a very interesting event and for numerous inspiring discussions that finally lead to an overview as presented in the present text. In particular, the authors thank Peter Arndt, Pierre Cartier, Javier Fres\'an, Florian Marty, Andrew Salch and Christophe Soul\'e for useful discussions and openly sharing the pieces of their unpublished works with us. The authors thank Susama Agarwala, Snigdhayan Mahanta, Jorge Plazas and Bora Yakinoglu for complementing the authors' research on $\Fun$ with many other valuable aspects.

\section{The building bricks of $\Fun$-geometries}\label{bricks}
	In this first part of the paper, we present an introduction to several different approaches to $\Fun$-geometry. Some of the original definitions have been reformulated in order to unify notation, simplify the exposition or make similarities with other notions  apparent. In what follows--unless explicitly mentioned otherwise--all rings will be assumed to be commutative rings with $1$. Monoids are commutative semi-groups with $1$ and will be written multiplicatively. Schemes are understood to be schemes over $\Z$, and by variety, we will mean a reduced scheme of finite type.

	\subsection{$\mathcal{M}$-schemes after Kato and Deitmar}\label{mschemes}
	In \cite{Deitmar2005}, Deitmar proposes the definition of a geometry over the field with one element by following ideas of Kato (cf.\ \cite{Kato1994}) of mimicking classical scheme theory but using the category $\cM$ of commutative monoids (abelian multiplicative semigroups with 1) in place of the usual category of (commutative and unital) rings. To a far extent, this idea leads to a theory that is formally analogous to algebraic geometry.

	Given a monoid $A$, an \dtext{ideal} of $A$ is a subset $\mathfrak{a}$ such that $\mathfrak{a}A\subseteq \mathfrak{a}$. An ideal $\mathfrak{a}$ is \dtext{prime} if whenever $xy\in \mathfrak{a}$, then either $x\in \mathfrak{a}$ or $y\in \mathfrak{a}$ (or equivalently, if the set $A\setminus\mathfrak{a}$ is a submonoid of $A$). The set $\spec A$ of all the prime ideals of the monoid $A$ can be endowed with the \dtext{Zariski topology}, for which a set $V\subseteq \spec A$ is closed if there is an ideal $\mathfrak{a}\subseteq A$ such that $V = V(\mathfrak{a}) = \{\mathfrak{q}\in \spec A|\ \mathfrak{a}\subseteq\mathfrak{q}\}$. A commutative monoid always contains a \dtext{minimal prime ideal}, the empty set, and a \emph{unique} \dtext{maximal prime ideal}, $\mathfrak{m}_A := A\setminus A^{\times}$, where $A^{\times}$ is the group of units of $A$. For any submonoid $S\subseteq A$ the \dtext{localization of $A$ at $S$} is defined by $S^{-1}A = \{\frac{a}{s}|\ a\in A, s\in S\}/\sim$, where $a/s \sim b/t$ if there exists some $u\in S$ such that $uta = usb$. The localization $\Quot A:= A^ {-1}A$ is called the \dtext{total fraction monoid} of $A$. The monoid $A$ is \dtext{integral} (or \dtext{cancellative}) if the natural map $A \to \Quot A$ is injective. For each prime ideal $\mathfrak{p}\subseteq A$  we construct the \dtext{localization at $\mathfrak{p}$} as $A_{\mathfrak{p}} := (A\setminus \mathfrak{p})^{-1} A$.

	Let $X= \spec A$ be endowed with the Zariski topology. Given an open set $U\subseteq X$, define
	\[
		\O_X(U) := \left\{s:U \to \bdu_{\mathfrak{p}\in U} A_{\mathfrak{p}}|\ s(\mathfrak{p})\in A_{\mathfrak{p}}\ \text{and $s$ is loc. a quotient of elements in $A$}   \right\}
	\]
	where $s$ is locally a quotient of elements in $A$ if $s(\mathfrak{q})=a/f$ for some $a,f \in A$ with $f\notin \mathfrak{q}$ for all $\mathfrak{q}$ in some neighborhood of $\mathfrak{p}$. We call $\O_X$ the \dtext{structure sheaf} of $X$. The \dtext{stalks} are $\O_{X,\mathfrak{p}} := \varinjlim_{\mathfrak{p}\in U} \O_X(U)\cong A_{\mathfrak{p}}$. Taking global sections yields $\O_X(X) = \Gamma(\spec A, \O_X) \cong A$.

	A \dtext{monoidal space} is a pair $(X,\O_X)$ where $X$ is a topological space and $\O_X$ is a sheaf of monoids. A \dtext{morphism of monoidal spaces} is a pair $(f,f^\#)$ where $f:X\to Y$ is a continuous map and $f^\#:\O_Y\to f_{\ast}\O_X$ is a morphism of sheaves. The morphism $(f,f^\#)$ is \dtext{local} if for each $x\in X$, we have $(f_x^\#)^{-1}(\O_{X,x}^{\times}) = \O_{Y,f(x)}^{\times}$. For every monoid $A$ and $X=\spec A$, the pair $(X, \O_X)$ is a monoidal space, called an \dtext{affine $\cM$-scheme}, and every morphism of monoids $\vphi:A \to B$ induces a local morphism between the corresponding monoidal spaces. A monoidal space $(X,\O_X)$ is called a \dtext{$\cM$-scheme} (over $\Fun$) if for all $x\in X$ there is an open set $U$ containing $x$ such that $(U,{\O_X}_{|U})$ is an affine $\cM$-scheme.

Let $\Z[A]$ denote the semi-group ring of $A$. The base extension functor $-\ot_{\Fun}\Z$ that sends $\spec A$ to $\Spec \mathbb Z[A]$ has a right-adjoint given by the forgetful functor from rings to monoids (\cite[Theorem 1.1]{Deitmar2005}). Both functors extend to functors between $\cM$-schemes and schemes over $\mathbb Z$ (\cite[section 2.3]{Deitmar2005}). We will often write $X_\Z$ for $X\otimes_{\F_1}\Z$.

	
	\subsection{Schemes over $\Fun$ in the sense of To\"en and Vaqui\'e}
\label{tvschemes}

To\"en and Vaqui\'e introduce in their paper \cite{Toen2008} $\Fun$-schemes as functors on monoids that can be covered by representable functors. To avoid confusion with the other notions of $\Fun$-schemes in the present text, we will call the $\Fun$-schemes after To\"en and Vaqiu\'e \emph{``TV-schemes''}. We will explain the notion of a TV-scheme in this section.

As in the previous section, we let $\cM$ be the category of monoids. Then the \emph{category of affine TV-schemes} is by definition the opposite category $\cM^\op$. If $A$ is a monoid in $\cM$, then we write $\Spec_{TV} A$ for the same object in the opposite category. The category of presheaves $\Psh(\cM^\op)$ on $\cM^\op$ consists of functors $\cM^\op\to\Sets$ as objects and natural transformations between them as morphisms. The affine TV-scheme $X=\Spec_{TV} A$ can be regarded as a presheaf $X:\cM^\op\to\Sets$ by sending $Y=\Spec_{TV} B$ to $X(Y)=\Hom(Y,X)=\Hom(A,B)$. This defines an embedding of categories $\cM^\op\hookrightarrow\Psh(\cM^\op)$.

For a monoid $A$, let $\MMod{A}$ be the category of sets with $A$-action together with equivariant maps. A homomorphism $f:A\to B$ of monoids is \emph{flat} if the induced functor
	\[
		-\otimes_AB: \ \MMod{A} \ \longrightarrow \ \MMod{B} 
	\]
commutes with finite limits.

Let $X=\Spec_{TV} A$ and $Y=\Spec_{TV} B$. A morphism $\varphi:Y\to X$ in $\cM^\op$ is called \emph{Zariski open} if the dual morphism $f: A\to B$ is an flat epimorphism of finite presentation. A \emph{Zariski cover} of $X=\Spec_{TV} A$ is a collection $\{X_i\to X\}_{i\in I}$ of Zariski open morphisms in $\cM^\op$ such that there is a finite subset $J\subset I$ with the property that the functor
	\[
		\prod_{j\in J} -\otimes_AA_j: \ \MMod{A} \ \longrightarrow \ \prod_{j\in J} \MMod{A_j}
	\]
is conservative. This defines the \emph{Zariski topology} on $\Spec A$. The category $\Sh(\cM^\op)$ of sheaves is the full subcategory of $\Psh(\cM^\op)$ whose objects satisfy the sheaf axiom for the Zariski topology. To\"en and Vaqui\'e show that $\Spec A$ is indeed a sheaf for every monoid $A$.

Let $\cF:\cM^\op\to\Sets$ be a subsheaf of an affine TV-scheme $X=\Spec A$. Then $F\subset X$ is \emph{Zariski open} if there exists a family $\{X_i\to X\}_{i\in I}$ of Zariski open morphisms in $\cM^\op$ such that $\cF$ is the direct image of 
$$\coprod_{i\in I} X_i\to X$$
as a sheaf. Let $\cF$ be a subsheaf of a sheaf $\cG:\cM^\op\to\Sets$. Then $\cF\subset \cG$ is \emph{Zariski open} if for all affine TV-schemes $X=\Spec A$, the natural transformation $\cF\times_\cG X\to X$ is a monomorphism in $\Sh(\cM^\op)$ with Zariski open image.

Let $\cF:\cM\to\Sets$ be a functor. An \emph{open cover of $\cF$} is a collection $\{X_i\hookrightarrow \cF\}_{i\in I}$ of Zariski open subfunctors such that 
$$\coprod_{i\in I}X_i\to\cF$$
is an epimorphism in $\Sh(\cM^\op)$.

A \emph{TV-scheme} is a sheaf $\cF:\cM^\op\to\Sets$ that has an open cover by affine TV-schemes.

Given an affine TV-scheme $X=\Spec_{TV} A$, we define its base extension to $\Z$ as the scheme $X_\Z=\Spec \Z[A]$. This extends to a base extension functor
$$ -\otimes_\Fun\Z: \ \{\text{TV-schemes}\} \ \longrightarrow \ \{\text{schemes}\}. $$

\begin{remark}
 In their paper \cite{Toen2008}, To\"en and Vaqui\'e define schemes w.r.t.\ any \emph{cosmos}, that is, a symmetric closed-monoidal category that admits small limits and small colimits. It is possible to speak of monoids in such a category. In the case of the category $\Sets$, we obtain $\cM$ as monoids and the schemes w.r.t.\ $\Sets$ are TV-schemes as described above. The category of $\Z$-modules yields commutative rings with $1$ as monoids, and the schemes w.r.t.\ the category of $\Z$-modules are schemes in the usual sense. There are several other interesting categories obtained by this construction that can be found in \cite{Toen2008}, also cf.\ Remark \ref{rem_salch}.
\end{remark}



\subsection{Haran's non-additive geometry} \label{hschemes}
In \cite{Haran2007}, Haran proposes an ``absolute geometry'' in which the field with one element and rings over $\Fun$ are realized as certain categories.

Haran's definition of the field with one element is the category whose objects are pointed finite sets and together with morphisms that are maps $\varphi:X\to Y$ such that $f_{|X\setminus \vphi^{-1}(0_Y)}$ is injective (where $0_Y$ is the base point of $Y$).

Neglecting the base point yields an equivalence of this category with the category $\F$ of finite sets together with partial bijections, i.e.\ the morphism set from $X$ to $Y$ is 
	\[
		\F_{Y,X} = \Hom_{\F}(X,Y) := \{\vphi:U\overset{\sim}{\lto} V|\ U\subseteq X, V\subseteq Y\}.
	\]

The category $\F$ is endowed with two functors $\oplus$ (disjoint union) and $\ot$ (cartesian product). These functors induce two structures $(\F,\oplus,[0]=\varnothing)$ and $(\F,\ot, [1]=\{\ast\})$ of symmetric monoidal categories on $\F$. An \dtext{$\F$-ring} is a category $A$ with objects $\Ob(A)=\Ob(\F)$ being finite sets, endowed with a faithful functor $\F\to A$ that is the identity on objects, and two functors $\oplus, \ot:A\times A \to A$ extending the ones in $\F$ and each making $A$ into a symmetric monoidal category. For example, every ring together with the matrix algebras with coefficients in the ring is an $\F$-ring. This embeds the category of rings into the category of $\F$-rings.

An \dtext{ideal} of an $\F$-ring $A$ is a collection of subsets $\{I_{Y,X}\subseteq A_{Y,X}\}_{X,Y\in |\F|}$ that is closed under the operations $\circ$, $\oplus$, $\ot$, i.e 
	\[
		A\circ I \circ A \subseteq I,\qquad A\ot I \subseteq I\quad\text{and}\quad I\oplus I \subseteq I.
	\]
An ideal $\mathfrak{a}\subseteq A$ is called \dtext{homogeneous} if it is generated by $\mathfrak{a}_{[1],[1]}$. A subset $\mathfrak{A}\subseteq A_{[1],[1]}$ is called an \dtext{H-ideal} if for all $a_1,\dotsc,a_n\in \mathfrak{A}$, $b\in A_{[1],[n]}$, $b'\in A_{[n],[1]}$ we have $b\circ(a_1\oplus\dotsb\oplus a_n)\circ b'\in \mathfrak{A}$. If $\mathfrak{a}$ is a homogeneous ideal, then $\mathfrak{a}_{[1],[1]}$ is an H-ideal, and conversely every H-ideal $\mathfrak{A}$ generates a homogeneous ideal $\mathfrak{a}$ such that $\mathfrak{a}_{[1],[1]}=\mathfrak{A}$. An H-ideal $\mathfrak{p}\subseteq A_{[1],[1]}$ is \dtext{prime} if its complement $A_{[1],[1]}\setminus \mathfrak{p}$ is multiplicatively closed. The set $\Spec A$ of all prime H-ideals can be endowed with the Zariski topology in the usual way. Localization of an $\F$-ring $A$ with respect to a multiplicative subset $S\subseteq A_{[1],[1]}$ also works exactly as classical localization theory for commutative rings, with the localization functor having all the nice properties we might expect. We will denote by $ A_{\mathfrak{p}}=\left(A_{[1],[1]}\setminus\mathfrak{p}\right)^{-1}A$ the localization of $A$ with respect to the complement of a prime $H$-ideal $\mathfrak{p}$.

As in Deitmar's geometry, we can use this localization theory to build structure sheaves. Let $A$ be an $\F$-ring and $U\subseteq \Spec A$ an open set. For $X, Y\in |\F|$ let $\O_A(U)_{Y,X}$ denote the set of functions
	\[
		s: U \lto \bdu_{\mathfrak{p}\in U} (A_{\mathfrak{p}})_{Y,X}
	\]
such that $s(\mathfrak{p})\in (A_\mathfrak{p})_{Y,X}$ and $s$ is locally a fraction, i.e.\ for all $\mathfrak{p}\in U$, there is a neighborhood $V$ of $\mathfrak{p}$ in $U$, an $a\in A_{Y,X}$ and an $f\in A_{[1],[1]}\setminus \bdu_{\mathfrak{q}\in V}\mathfrak{q}$ such that $s(\mathfrak{q}) = a/f$ for all $\mathfrak{q} \in V$. This construction yields a sheaf $\O_A$ of $\F$-rings, which is called the \dtext{structure sheaf} of $A$. For each H-prime $\mathfrak{p}$ the stalk $\O_{A,\mathfrak{p}}:=\varinjlim_{\mathfrak{p}\in U} \O_A(U)$ is isomorphic to $A_{\mathfrak{p}}$. Moreover, the $\F$-ring $\Gamma(\Spec A, \O_A)$ of global sections is isomorphic to $A$.

An \dtext{$\F$-ringed space} is a pair $(X,\O_X)$ where $X$ is a topological space and $\O_X$ is a sheaf of $\F$-rings. An $\F$-ringed space $(X,\O_X)$ is \emph{$\F$-locally ringed} if for each $p\in X$ the stalk $\O_{X,p}$ is local, i.e.\ contains a unique maximal H-ideal $\mathfrak{m}_{X,p}$. A \dtext{(Zariski) $\F$-scheme} is an $\F$-locally ringed space $(X,\O_X)$ such that there is an open covering $X=\bigcup_{i\in I} U_i$ for which the canonical maps $P_i:(U_i,{\O_{X}}_{|U_i})\to \Spec\O_X(U_i)$ are isomorphisms of $\F$-locally ringed spaces. The category of \dtext{$\F$-schemes} is the category of inverse systems (or pro-objects) in the category of Zariski $\F$-schemes.

For $\F$-rings $A$, $B$, $C$ and morphisms $A\to B$ and $A\to C$, we can construct the \emph{(relative) tensor product $B\circledcirc_A C$}. Using this construction, it follows that the category of (Zariski) $\F$-schemes contains fibered sums. In particular, we can take $A=\F$ and $C=\F(\Z)$, in order to obtain an extension of scalars from $\F$ to $\F(\Z)$; however, the category of $\F(\Z)$-algebras is not equivalent to the category of rings (this is pretty much due to the existence of $\F$-rings which are not matrix rings, cf.\ \cite{ArndtUN} for details), so this functor does not provide an extension of scalars from $\F$-schemes to usual schemes.

The embedding of the category of rings inside the category of $\F$-rings show that every scheme can be regarded as an $\F$-scheme. A similar construction allows to produce an $\F$-ring out of a monoid, providing a relation with $\cM$-schemes, cf.\ section \ref{mtoh} for further details. It is worth noting that all the examples mentioned here are what Haran calls \emph{rings of matrices}; there are examples (cf.\ \cite[\S 2.3, Examples 3 and 5]{Haran2007} of more exotic $\F$-rings that are not rings of matrices. Haran succeeds (cf.\ \cite[\S 6.3]{Haran2007}) in defining the completion $\overline{\Spec\Z}$ of the spectrum of $\Z$, which is one step in Deninger's program to prove the Riemann hypothesis.

	
\subsection{Durov's generalized schemes} \label{durov}

In \cite{Durov2007}, Durov introduces a generalization of classical algebraic geometry, into that Arakelov geometry fits naturally. As a byproduct of his theory of generalized rings, he obtains a model for the field with one element and a notion of a geometry over $\Fun$. Here, we outline Durov's construction briefly. For full details, cf.\ \cite{Durov2007}, or consider \cite{Fresan2009} for a summary.

Any ring $R$ can be realized as the endofunctor in the category of sets that maps a set $X$ to $R^X$, the free $R$-module generated by $X$. The ring multiplication and unit translate into properties of this functor, making it into a \dtext{monad} that commutes with filtered direct limits (cf.\ for instance \cite{Street1972} or \cite[Chapter VI]{MacLane1998}). Motivated by this fact, Durov defines a \dtext{generalized (commutative) ring} as an monad in the category of sets that is commutative (cf.\ \cite[\S 5.1]{Durov2007} for the precise notion of commutativity) and commutes with filtered direct limits. If the set $A([0])=A(\varnothing)$ is not empty, $A$ is said to \emph{admit a constant}, or we say that $A$ is a \emph{generalized ring with zero} (cf.\ \cite[\S 5.1]{Durov2007}). There is a natural notion of a \dtext{module} over a monad, which allows to construct the category of modules over any generalized ring $A$. Every generalized ring $A$ has an underlying monoid $|A|:=A([1])$, so we can define a \dtext{prime ideal} of $A$ as any $A$-submodule $\mathfrak{p}$ of $|A|$ such that the complement $|A|\setminus\mathfrak{p}$ is a multiplicative system. The set $\Spec A$ of all the prime ideals in $A$ can be endowed with the Zariski topology in the usual way. 

The notions of localization, presheaves and sheaves of generalized rings are defined analogously to the usual theory of schemes resp.\ to the theory of the former sections. We can talk about \dtext{generalized ringed spaces} $(X,\O_X)$ consisting of a topological space $X$ with a sheaf of generalized rings $\O_X$. A generalized ringed space is \dtext{local} if for every $p\in X$ the generalized ring $\O_{X,p}$ has a unique maximal ideal. Every generalized ring defines a locally generalized ringed space $(\Spec A, \O_A)$, which is called the \dtext{generalized affine scheme} associated to $A$. A \dtext{generalized scheme} is then a locally generalized ringed space which is locally isomorphic to a generalized affine scheme. In this setting, the category of $\Fun$-schemes consists precisely of those generalized schemes $(X,\O_X)$ for that the set $\Gamma(X,\O_X(\boldsymbol{0}))$ is not empty, also called \dtext{generalized schemes with zero} (cf.\ \cite[\S 6.5.6]{Durov2007}).

Examples of generalized schemes in the sense of Durov include usual schemes via the aforementioned realization of a ring as a generalized ring, as well as schemes over the spectrum of monoids or semi-rings, in particular in Durov's theory it is possible to speak about schemes over the natural numbers, the completion $\Z_\infty$ or the tropical semi-ring $\mathbb{T}$. As Haran does in the context of his non-additive geometry, Durov defines the completion $\overline{\Spec\Z}$ of the spectrum of $\Z$ as a generalized scheme with zero. However, Durov's construction forces the fibered product $\overline{\Spec\Z}\times_{\Fun} \overline{\Spec\Z}$ to be isomorphic again to $\overline{\Spec\Z}$, making it unsuitable to pursue Deninger's program.


\subsection{Varieties over $\Fun$ in the sense of Soul\'e and its variations}
\label{svarietiesandco}

The first suggestion of a category that realizes a geometry over $\Fun$ was given by Soul\'e in \cite{Soule2004}. This notion found several variations (\cite{Soule2009}, \cite{Connes2008}) that finally lead to the notion of $\Fun$-schemes as given by Connes and Consani in \cite{Connes2009}, cf.\ the following section. We do not give the formal definitions from \cite{Soule2004}, \cite{Soule2009} and \cite{Connes2008} in its completeness, but try to give an overview of how ideas developed.

\subsubsection{S-varieties}
\label{svarieties}
We begin with the notion of varieties over $\Fun$ as given in \cite{Soule2004}. To avoid confusion, we will call these varieties over $\Fun$ \emph{``S-varieties''}. For a precise definition, cf.\ \cite{Soule2004} or \cite{LL}.

An \emph{affine S-variety} $X$ consists of
\begin{itemize}
 \item a functor $\uline X: \{\text{finite flat rings}\} \to \{\text{finite sets}\}$,
 \item a complex (non necessarily commutative) algebra $\cA_X$,
 \item an affine scheme $X_\Z=\Spec A$ of finite type,
 \item a natural transformation $\ev_X:\uline X\Rightarrow\Hom(\cA_X,-\otimes_\Z\C)$,
 \item an inclusion of functors $\uline\iota:\uline X\Rightarrow\Hom(A,-)$ and
 \item an injection $\iota_\C:A\otimes_\Z\C\hookrightarrow\cA_X$
\end{itemize}
such that the diagram
\begin{equation}\label{S-diagram}
\xymatrix{\uline X(R)\ar[rr]^{\uline\iota(R)}\ar[d]_{\ev_X(R)}&&\Hom(A,R)\ar[d]^{-\otimes_\Z\C}\\\Hom(\cA_X,R\otimes_\Z\C)\ar[rr]^{\iota_\C^\ast} &&\Hom(A\otimes_\Z\C,R\otimes_\Z\C)} 
\end{equation}
commutes for all finite flat rings $R$ and such that a certain universal property is satisfied. This universal property characterizes $X_\Z$ together with $\uline \iota$ and $\iota_\C$ as the unique extension of the triple $(\uline X,\cA_X,\ev_X)$ to an affine S-variety. We define $X_\Z$ as the \emph{base extension of $X$ to $\Z$}.

A \emph{morphism $X\to Y$ between affine S-varieties} consists of
\begin{itemize}
 \item a natural transformation $\uline X\to\uline Y$,
 \item a $\C$-linear map $\cA_Y\to\cA_X$ and
 \item a morphism $X_\Z\to Y_\Z$ of schemes
\end{itemize}
such that they induce a morphism between the diagrams \eqref{S-diagram} corresponding to $X$ and $Y$.

The idea behind the definition of global S-varieties is to consider functors on affine S-varieties. More precisely, an \emph{S-variety} $X$ consists of 
\begin{itemize}
 \item a functor $\uline X: \{\text{affine S-varieties}\} \to \{\text{finite sets}\}$,
 \item a complex algebra $\cA_X$,
 \item a scheme $X_\Z$ of finite type with global sections $A$,
 \item a natural transformation $\ev_X:\uline X\Rightarrow\Hom(\cA_X,\cA_{(-)})$,
 \item an inclusion of functors $\uline\iota:\uline X\Rightarrow\Hom((-)_\Z,X_\Z)$ and
 \item an injection $\iota_\C:A\otimes_\Z\C\hookrightarrow\cA_X$
\end{itemize}
such that the diagram
\begin{equation*}
\xymatrix{\uline X(V)\ar[rr]^{\uline\iota(V)}\ar[d]_{\ev_X(V)}&&\Hom(V_\Z,X_\Z)\ar[d]^{\substack{\text{global}\\ \text{sections}}}\\ \Hom(\cA_X,\cA_V)\ar[rr]^{\iota_\C^\ast} &&\Hom(A\otimes_\Z\C,\cA_V)} 
\end{equation*}
commutes for all finite flat rings $R$ and such that a certain universal property is satisfied. This universal property plays the same r\^ole as the universal property in the definition of an affine S-variety. We define $X_\Z$ as the base extension of $X$ to $\Z$. Morphisms of S-varieties are defined analogously to the affine case.

A remarkable property of these categories is that the dual of the category of finite flat rings embeds into the category of affine S-varieties and that the category of affine S-varieties embeds into the category of S-varieties. The essential image of the latter embedding are those S-varieties whose base extension to $\Z$ is an affine scheme. Furthermore, a system of affine S-varieties ordered by inclusions that is closed under intersections can be glued to an S-variety.

In his paper, Soul\'e constructs for every smooth toric variety an S-variety whose base extension to $\Z$ is isomorphic to the toric variety. This result was extended in \cite[Theorem\ 3.11]{LL}, also cf.\ section \ref{torifiedtos}. 

The paper contains a definition of a zeta function for those S-varieties $X$ that admit a polynomial counting function, i.e.\ that the function $q\mapsto\# X_\Z(\F_q)$ on prime powers $q$ is given by a polynomial in $q$ with integer coefficients. This provides, up to a factor $2\pi$, a first realization of Deninger's motivic zeta functions of $h^0$ and $h^2$ (as in the introduction) as zeta functions of the \emph{``absolute point $\Spec\Fun$''} and the \emph{``affine line over $\Fun$''}.

\subsubsection{S${}^\ast$-varieties}
\label{sstarvarieties}

In \cite{Soule2009}, Soul\'e describes the first modification of this category. The idea is to exchange the functor on finite flat rings by a functor on finite abelian groups. Correspondingly, the functors $\Hom(\cA_X,-\otimes_\Z\C)$ and $\Hom(A,-)$ in the definition of an affine S-variety are exchanged by $\Hom(\cA_X,\C[-])$ and $\Hom(A,\Z[-])$, respectively, where $\Z[D]$ and $\C[D]$ are the group algebras of a finite abelian group $D$. We call the objects that satisfy the definition of an affine S-variety with these changes \emph{affine S${}^\ast$-varieties}. Morphisms between affine S${}^\ast$-varieties and the base extension to $\Z$ is defined analogously to the case of affine S-varieties. Also the definition of \emph{S${}^\ast$-varieties} and morphisms between S${}^\ast$-varieties are defined in complete analogy to the case of S-varieties. 

Soul\'e proves similar results for S${}^\ast$-varieties as for S-varieties: Affine S${}^\ast$-varieties are S${}^\ast$-varieties. A system of affine S${}^\ast$-varieties ordered by inclusions that is closed under intersections can be glued to an S${}^\ast$-variety. Every smooth toric variety has a model as an S${}^\ast$-variety (see also section \ref{torifiedtos}). The definition of zeta functions transfers to this context.

\subsubsection{CC-varieties}
\label{ccvarieties}

In \cite{Connes2008}, Connes and Consani modify the notion of an S${}^\ast$-variety in the following way. They endow the functor $\uline X$ from finite abelian groups to finite sets with a grading, i.e.\ for every finite abelian group $D$, the set $\uline X(D)$ has a decomposition $\uline X(D)=\coprod_{n\in\N} \uline X^n(D)$ into a disjoint union of sets. Secondly, they exchange the complex algebra by a complex variety. We call the objects that satisfy the definition of an affine S${}^\ast$-variety with these changes \emph{affine CC-varieties}. Morphisms between affine CC-varieties and the base extension to $\Z$ are defined analogously to the case of affine S-varieties and S${}^\ast$-varieties. 

The category of affine CC-varieties is embedded into a larger category that plays the r\^ole of the category of locally ringed spaces in the theory of schemes. \emph{CC-varieties} are defined as those objects in the larger category that admit a cover by affine CC-varieties.

The definition of zeta functions transfers to this context. Connes and Consani show certain examples of CC-schemes, where the zeta function can be read off from the graded functor of the CC-scheme.

The main application of this paper is the construction of models of split reductive groups as CC-varieties over ``$\F_{1^2}$''. This is a first construction of objects in $\Fun$-geometry that contributes to Tits' ideas from \cite{Tits1957}. These results were extended in \cite{LL}; in particular, they also hold in the context of S-varieties and S${}^\ast$-varieties, cf.\ section \ref{torifiedtos}. However, these categories are not suitable to define a group law for any split reductive group over $\Fun$, cf.\ \cite[p.\ 25]{Connes2008} and \cite[Section 6.1]{LL}. For how this can be done in the context of CC-scheme, see section \ref{algebraic_groups} of the present text.


\subsection{Schemes over $\Fun$ in the sense of Connes and Consani}
\label{ccschemes}

In their paper \cite{Connes2009}, Connes and Consani merge Soul\'e's idea and its variations with Kato's resp.\ Deitmar's monoidal spaces and To\"en and Vaqui\'e's sheaves on monoids. Roughly speaking, a scheme over $\Fun$ in the sense of Connes and Consani, which we call \emph{``CC-scheme''} to avoid confusion, is a triple consisting of a locally representable functor on monoids, a scheme and an evaluation map. Unlike Kato/Deitmar and To\"en-Vaqui\'e, all monoids are considered together with a base point as in Haran's theory (\cite{Haran2007})--the locally representable functor is a functor from $\cM_0$ (see section \ref{hschemes}) to $\Sets$.

We make these notions precise. First of all, we can reproduce all steps in the construction of $\cM$-schemes as in section \ref{mschemes} to define $\cM_0$-schemes. Namely, an \emph{ideal} of a monoid $A$ with $0$ is a subset $I\subset A$ containing $0$ such that $IA\subset I$. A \emph{prime ideal} is an ideal $\p\subset A$ such that $A-\p$ is a submonoid of $A$. As in section \ref{mschemes}, we can define \emph{localizations} and the \emph{Zariski topology} on $\Spec_{\cM_0} A=\{\text{prime ideals of }A\}$, the \emph{spectrum of $A$}.

A \emph{monoidal space with $0$} is a topological space together with a sheaf in $\cM_0$. Together with its Zariski topology and its localizations, $\Spec A$ is a monoidal space with $0$. A \emph{geometric $\cM_0$-scheme} is a monoidal space with $0$ that has an open covering by spectra of monoids with $0$. 

Similar to section \ref{tvschemes}, $\cM_0$-schemes are defined as locally representable functors on $\cM_0$. In detail, an \emph{$\cM_0$-functor} is a functor from $\cM_0$ to $\Sets$. Every monoid $A$ with $0$ defines an $\cM_0$-functor $X=\Spec_{\cM_0} A$ by sending a monoid $B$ with $0$ to the set $X(B)=\Hom(\Spec_{\cM_0} B,\Spec_{\cM_0} A)=\Hom(A,B)$. An \emph{affine $\cM_0$-scheme} is a $\cM_0$-functor that is isomorphic to $\Spec_{\cM_0}A$ for some monoid $A$ with $0$.

A subfunctor $Y\subset X$ of an $\cM_0$-functor $X$ is called an \emph{open subfunctor} if for all monoids $A$ with $0$ and all morphisms $\varphi:Z=\Spec_{\cM_0}A\to X$, there is an ideal $I$ of $A$ such that for all monoids $B$ with $0$ and for all $\rho\in Z(B)=\Hom(A,B)$,
$$ \varphi(\rho)\in Y(B)\subset X(B) \iff \rho(I)B=B. $$
An \emph{open cover of $X$} is a collection $\{X_i\to X\}_{i\in I}$ of open subfunctors such that the map $\coprod_{i\in I}X_i(H)\to X(H)$ is surjective for every monoid $H$ that is a union of a group with $0$. An \emph{$\cM_0$-scheme} is an $\cM_0$-functor that has an open cover by affine $\cM_0$-schemes. 

If $X$ is an $\cM_0$-scheme, then the \emph{stalk at $x$} is $\cO_{X,x}=\lim\limits_\to X(U)$, where $U$ runs through all open neighborhoods of $x$. A \emph{morphism of $\cM_0$-schemes $\varphi: X\to Y$} is a natural transformation of functors that is \emph{local}, i.e.\ for every $x\in X$ and $y=\varphi(x)$, the induced morphism $\varphi^\sharp_x: \cO_{Y,y}\to\cO_{X,x}$ between the stalks satisfies that $(\varphi^\sharp_x)^{-1}(\cO_{X,x}^\times)=\cO_{Y,y}^\times$.

Connes and Consani claim that the $\cM_0$-functor $\Hom(\Spec_{\cM_0}-,X)$ is a $\cM_0$-scheme for every geometric $\cM_0$-scheme $X$ and that, conversely, every $\cM_0$-scheme is of this form (\cite[Prop.\ 3.16]{Connes2009}). For this reason, we shall make no distinction between geometric $\cM_0$-schemes and $\cM_0$-schemes from now on.

The association 
$$ A\to \bigl(Z[A]/1\cdot 0_A-0_{\Z[A]}\bigr) $$
where $0_A$ is the zero of $A$ and $0_{\Z[A]}$ is the zero of $\Z[A]$ extends to the base extension functor
$$ -\otimes_{\cM_0}\Z: \ \{\cM_0-\text{schemes}\} \ \longrightarrow \ \{\text{schemes}\}. $$
We denote the base extension of an $\cM_0$-scheme $X$ to $\Z$ by $X_\Z=X\otimes_{\cM_0}\Z$.

The ideas from sections \ref{svarieties}--\ref{ccvarieties} find now a simplified form. A \emph{CC-scheme} is a triple $(\widetilde X,X,\ev_X)$ where $\widetilde X$ is an $\cM_0$-scheme, $X$ is a scheme, viewed as a functor on the category of rings, and $\ev_X: \widetilde X_\Z\Rightarrow X$ is a natural transformation that induces a bijection $\ev_X(k):\widetilde X_\Z(k)\to X(k)$ for every field $k$. The natural notion of a \emph{morphism between CC-schemes $(\widetilde Y,Y,\ev_Y)$ and $(\widetilde X,X,\ev_X)$} is a pair $(\widetilde\varphi,\varphi)$ where $\widetilde\varphi:\widetilde Y\to\widetilde X$ is a morphism of $\cM_0$-schemes and $\varphi:Y\to X$ is a morphism of schemes such that the diagram
$$ \xymatrix{\widetilde X_\Z\ar[rr]^{\widetilde\varphi_\Z}\ar[d]_{\ev_X}&&\widetilde Y_\Z\ar[d]^{\ev_Y}\\ X\ar[rr]^\varphi&&Y} $$
commutes. The base extension of $\cX=(\widetilde X,X,\ev_X)$ to $\Z$ is $\cX_\Z=X$.

\begin{remark}
\label{rem_salch}
The definition of an $\cM_0$-scheme is in the line of thoughts of To\"en and Vaqui\'e's theory as described in section \ref{tvschemes}. Indeed, monoid objects in the category of pointed sets are monoids with $0$. It seems to be likely, but it is not obvious, that the notion a scheme w.r.t.\ the category of pointed sets are $\cM_0$-schemes. Andrew Salch is working on making this precise. Furthermore, Salch constructs a cosmos such that monoid objects on this category correspond to triples $(\widetilde X,X,e_X)$ where $\widetilde X$ is an $\cM_0$-functor, $X$ is a functor on rings and $e_X:\widetilde X_\Z\Rightarrow X$ is a natural equivalence. 

It is however questionable whether CC-schemes have an interpretation as schemes w.r.t.\ some category in the sense of To\"en and Vaqui\'e since it is not clear what an affine CC-scheme should be. See Remark \ref{non-affine_torification} for a more detailed explanation.
\end{remark}

\subsection{Torified varieties}\label{torifsch}

Torified varieties and schemes were introduced by the authors of this text in \cite{LL} in order to establish examples of varieties over $\Fun$ in the sense of Soul\'e (cf.\ section \ref{svarieties}) and Connes-Consani (cf.\ section \ref{ccvarieties}).

A \dtext{torified scheme} is a scheme $X$ endowed with a \dtext{torification} $e_X:T\to X$, i.e.\ $T=\bdu_{i\in I} \Gm^{d_i}$ is a disjoint union of split tori and $e_X$ is a morphism of schemes such that the restrictions ${e_X}_{|\Gm^{d_i}}$ are immersions and $e_X(k):T(k)\to X(k)$ is a bijection for every field $k$. Examples of torified schemes include (cf.\ \cite[\S 1.3]{LL} or \cite{Lorscheid2009}):
	\begin{itemize}
	 	\item Toric varieties, with torification given by the orbit decomposition.
		\item Grassmann and Schubert varieties, where the torification is induced by the Schubert cell decomposition. 
		\item Split reductive groups, with torification coming from the Bruhat decomposition.
	\end{itemize}
	
Let $X$ and $Y$ be torified schemes with torifications $e_X:T=\bdu_{i\in I} \Gm^{d_i}\to X$ and $e_Y:S=\bdu_{j\in J} \Gm^{f_j}\to Y$, respectively. A \dtext{torified morphism $(X,T)\to (Y,S)$} consists of a pair of morphisms $\vphi:X\to Y$ and $e_{\vphi}:T\to S$ such that the diagram 
	\[
		\xymatrix{
			X \ar[rr]^{\vphi} && Y \\
			T \ar[u]^{e_X} \ar[rr]^{e_\vphi} && S \ar[u]_{e_Y}
		}
	\]
commutes and such that for every $i\in I$ there is a $j\in J$ such that the restriction ${e_{\vphi}}_{|\Gm^{d_i}}: \Gm^{d_i} \to \Gm^{f_j}$ is a morphism of algebraic groups.

A torified scheme $(X,T)$ is \dtext{affinely torified} if there is an affine open cover $\{U_j\}$ of $X$ that respects the torification, i.e.\ such that for each $j$ there is a subset $T_j=\bdu_{i\in I_j}\Gm^{d_j}$ such that the restriction ${e_X}_{|T_j}$ is a torification of $U_j$.	A torified morphism is \dtext{affinely torified} if there is an affine open cover $\{U_j\}$ of $X$ respecting the torification and such that the image of each $U_j$ is an affine subscheme of $Y$. An \emph{(affinely) torified variety} is an (affinely) torified scheme that is reduced and of finite type over $\Z$.

Toric varieties and split reductive groups are examples of affinely torified varieties, cf.\ \cite[\S 1.3]{LL}, whilst the torifications associated to Grassmann and Schubert varieties are in general not affine.


\subsection{$\Lambda$-schemes after Borger}
\label{lambdaschemes}

An approach that is of a vein different to all the other $\Fun$-geometries is Borger's notion of a $\Lambda$-scheme (see \cite{Borger2009}). A $\Lambda$-scheme is a scheme with an additional decoration, which is interpreted as descent datum to $\Fun$. In order to give a quick impression, we restrict ourselves to introduce only flat $\Lambda$-schemes in this text. 

A \emph{flat $\Lambda$-scheme} is a flat scheme $X$ together with a family $\Phi=\{\varphi_p:X\to X\}_{p\text{ prime}}$ of pairwise commuting endomorphism of $X$ such that the diagram
$$ \xymatrix{X\otimes_\Z\F_p\ar[d]\ar[rr]^{\Frob_p} && X\otimes_\Z\F_p\ar[d]\\ X\ar[rr]^{\varphi_p}&& X} $$
commutes for every primes number $p$. Here, $\Frob_p:X\otimes_\Z\F_p\to X\otimes_\Z\F_p$ is the Frobenius morphism of the reduction of $X$ modulo $p$. 

The base extension of this flat $\Lambda$-scheme is the scheme $X$. In particular every reduced scheme $X$ that has a family $\Phi$ as described above is flat.

Examples of $\Lambda$-schemes are toric varieties. For more examples, consider \cite[sections 2]{Borger2009} or section \ref{mtolambda} of this text.


\section{Paths and bridges}\label{bridges}

In this second part of the paper, we review and construct various functors between the different $\Fun$-geometries. In some central cases, we will describe functors in detail, in other cases, we will give a reference. We will also describe some work in progress by Arndt and Marty on such functors. All these functors are displayed in the large diagram in Figure \ref{bigpicture}.


\subsection{Toric varieties as $\cM$-schemes}
\label{torictom}

In this section, we recall the definition of a toric variety, describe the reformulation of toric varieties in terms of a fan of monoids as done by Oda (\cite{Oda1988}) and relate them to $\cM$-schemes as done by Kato (\cite{Kato1994}). For more details on toric varieties, consider \cite{Fulton1993} and \cite{Maillot2000}.

A \emph{(strictly convex and rational) cone (in $\R^n$)} is an additive semi-group $\tau\subset\R^n$ of the form $\tau=\sum_{i\in I} t_i\R_{\geq0}$ where  $\{t_i\}_{i\in I}\subset\Q^n\subset\R^n$ is a linearly independent set. A \emph {face $\sigma$ of $\tau$} is a cone of the form $\sigma=\sum_{i\in J} t_i\R_{\geq0}$ for some subset $J\subset I$. A \emph{fan} is a non-empty collection $\Delta$ of cones in $\R^n$ such that
\begin{enumerate}
 \item each face of a cone $\tau\in\Delta$ is in $\Delta$ and
 \item for all cones $\tau,\sigma\in\Delta$, the intersection $\tau\cap\sigma$ is a face of both $\tau$ and $\sigma$.
\end{enumerate}
In particular, the face relation makes $\Delta$ into a partially ordered set.

If $\tau$ is a cone in $\R^n$, we define $A_\tau$ as the intersection $\tau^\vee\cap\Z^n$ of the dual cone $\tau^\vee$ with the lattice $\Z^n\subset\R^n$. Since the generators $t_i$ of $\tau$ have rational coordinates, $A_\tau$ is a finitely generated (additively written) monoid. An inclusion $\sigma\subset\tau$ of cones induces an open embedding of schemes $\Spec\Z[A_\sigma]\hookrightarrow\Spec\Z[A_\tau]$. A \emph{toric variety (of dimension $n$)} is a scheme $X$ together with a fan $\Delta$ (in $\R^n$) such that
$$ X \ \simeq \ \lim_{\substack{\longrightarrow\\ \tau\in\Delta}} \; \Spec\Z[A_\tau]. $$

A \emph{morphism $\psi:\Delta\to \Delta'$ of fans} is map $\widetilde\psi$ between partially ordered sets together with a direct system of semi-group morphisms $\psi_\tau:\tau\to\widetilde\psi(\tau)$ (with respect to the inclusion of cones). The dual morphisms induce monoid homomorphisms $\psi_\tau^\vee: A_{\widetilde\psi(\tau)}\to A_\tau$. Taking the direct limit over the system of scheme morphisms $\Spec\Z[\psi_\tau^\vee]:\Spec\Z[A_\tau]\to \Spec\Z[A_{\widetilde\psi(\tau)}]$ yields a morphism $\varphi:X\to X'$ between toric varieties. Such a morphism is called a \emph{toric morphism}.

Note that a toric variety is determined by its fan and a toric morphism between toric varieties is determined by the morphism between the fans. This leads to a completely combinatorial description of the category of toric varieties in terms of monoids as follows (cf.\ \cite[section 9]{Kato1994}).

Recall the definition of the quotient group $\Quot A$ and of an integral monoid from section \ref{mschemes}. The monoid $A$ is \emph{saturated} if it is integral and if for all $a\in A$, $b\in \Quot A$ and $n>0$ such that $b^n=a$, we have that $b\in A$. A \emph{fan in $\Z^n$} is a collection $\Delta$ of (additive) submonoids $A\subset\Z^n$ such that
\begin{enumerate}
 \item all $A\in\Delta$ are finitely generated and saturated, $A^\times={1}$ and $\Z^n/\Quot A$ is torsion-free,
 \item for all $A\in\Delta$ and $\p\in\Spec_\cM A$, we have $A\setminus\p\in\Delta$, and
 \item for all $A,B\in\Delta$, we have $A\cap B=A\setminus\p=B\setminus\q$ for some $\p\in\Spec_\cM A$ and $\q\in\Spec_\cM B$.
\end{enumerate}
In particular, $\Delta$ is a diagram of monoids via the inclusions $A\setminus\p\hookrightarrow A$.

This yields the following alternative description: a toric variety is a scheme $X$ together with a fan $\Delta$ in $\Z^n$ such that 
$$ X \ \simeq \ \lim_{\substack{\longrightarrow\\ A\in\Delta}} \; \Spec\Z[A]. $$
Note that $\Spec\Z[A]$ is the base extension of the $\cM$-scheme $\Spec_\cM A$ to $\Z$, and that for
$$ \widetilde X \ = \ \lim_{\substack{\longrightarrow\\ A\in\Delta}} \; \Spec_\cM A, $$
$\widetilde X_\Z\simeq X$. We introduce some further definitions to state Theorem \ref{thm_torictom} where we follow ideas from \cite[section 4]{Deitmar2007}.

An $\cM$-scheme is \emph{connected} if its topological space is connected. An $\cM$-scheme is \emph{integral / of finite type / of exponent $1$} if all its stalks are integral / of finite type / (multiplicatively) torsion-free.

\begin{theorem}[{\cite[Thm.\ 4.1]{LL}}]
\label{thm_torictom} 
The association as described above extends to a functor 
$$ \mathcal K:\{\text{toric varieties}\} \ \longrightarrow \ \{\cM\text{-schemes}\} $$
that induces an equivalence of categories
	\[ 
		\mathcal K:\ \ \left\{\begin{array}{c}\text{toric varieties}\end{array}\right\}
           \ \ \ \stackrel\sim\longrightarrow \ \ \ 
           \left\{\begin{array}{c}\textrm{connected integral $\cM$-schemes}\\ 
                                  \textrm{of finite type and of exponent }$1$ \end{array}\right\}
	\]
with $-\otimes_{\mathbb F_1}\mathbb Z$ being its inverse.

\end{theorem}


\subsection{Comparison of $\cM$-schemes with TV-schemes}
\label{mtotv}

A project of Florian Marty is the investigation of the connection between $\cM$-schemes and TV-schemes. In the second arXiv version of his paper \cite{Marty2007}, one finds partial results. From a private communication, we get the following picture.

It is possible to associate to each affine TV-scheme $X=\Spec_{TV} A$ a topological space $Y$ whose locale is the locale of $X$ (defined by the Zariski open subsheaves as introduced in section \ref{tvschemes}) and such that $Y$ is homeomorphic to the topological space $\Spec_\cM A$ as considered in section \ref{mschemes}. More precisely, this construction yields a functor
$$ \cM_1: \ \{\text{affine TV-schemes}\} \ \longrightarrow \ \{\text{affine $\cM$-schemes}\} $$
that is an equivalence of categories with inverse $\cM_2:\Spec_\cM A \mapsto \Hom(A,-)$. Moreover, Marty also proves in \cite{Marty2007} that the notions of open immersions for $\cM$-schemes and TV-schemes coincide.

It seems to be likely, but it is not obvious, that the functors $\cM_1$ and $\cM_2$ extend to an equivalence of categories
$$ \xymatrix{\{\text{TV-schemes}\}\ar@{-->}@<0,5ex>[rr]^{\cM_1} && \{\text{$\cM$-schemes}\} \ar@{-->}@<0,5ex>[ll]^{\cM_2}}. $$


\subsection{$\cM$-schemes as $\cM_0$- and CC-schemes}
\label{mtocc}

The category of $\cM$-schemes embeds into the category of CC-schemes. We proceed in two steps.

Firstly, we construct a functor from $\cM$-schemes to $\cM_0$-schemes. Consider the fully faithful functor $\cM\stackrel{+0}\longrightarrow\cM_0$ that associates to a monoid $A$ the monoid $A_{+0}=A\cup\{0\}$ with $0$ whose multiplication is extended by $0\cdot a=0$ for every $a\in A_{+0}$. This induces a faithful functor
	\[
		\{\text{monoidal spaces}\} \ \stackrel{+0}\longrightarrow \ \{\text{monoidal spaces with $0$}\}, 
	\]
that we denote by the same symbol, by composing the structure sheaf $\cO_X$ of a monoidal space $X$ with $\cM\stackrel{+0}\longrightarrow\cM_0$. We reason in the following that this functor restricts to a functor from the category of $\cM$-schemes to the category of $\cM_0$-schemes.

The prime ideals of a monoid $A$ (cf.\ section \ref{mschemes}) coincide with the prime ideals of $A_{+0}$ (cf.\ section \ref{ccschemes}) since we asked a prime ideal of a monoid with $0$ to contain $0$. Recall that localizations, the Zariski topology, monoidal spaces (without and with $0$) for $\cM$-schemes and $\cM_0$-schemes are defined in complete analogy. Thus the above functor restricts to a faithful functor
	\[ 
		\{\text{$\cM$-schemes}\} \ \stackrel{+0}\longrightarrow \ \{\text{$\cM_0$-schemes}\}.
	\]

Secondly, we define the functor 
	\[ 
		\cF: \{\text{$\cM_0$-schemes}\} \ \longrightarrow \ \{\text{CC-schemes}\}
	\]
as sending an $\cM_0$-scheme $X$ to the CC-scheme $(X,X_\Z,\id_{X_\Z})$. A morphism $\varphi:Y\to X$ of $\cM_0$-schemes defines the morphism $(\varphi,\varphi_\Z)$ of CC-schemes. The condition that
$$ \xymatrix{Y_\Z \ar[rr]^{\varphi_\Z}\ar[d]_{\id_{Y_\Z}} && X_\Z\ar[d]^{\id_{X_\Z}} \\ Y_\Z \ar[rr]^{\varphi_\Z} && X_\Z} $$
commutes is trivially satisfied and shows that $\cF$ is faithful. The composition of the functors $+0$ and $\cF$ yields a functor
$$ \{\text{$\cM$-schemes}\} \ \longrightarrow \ \{\text{CC-schemes}\}. $$


\subsection{$\mathcal{M}$- and $\mathcal{M}_0$-schemes as $\F$-schemes}
\label{mtoh}

Let $M$ be a monoid with 0. We can define an $\F$-ring $\F\langle M\rangle$ by setting $\F\langle M\rangle_{Y,X} = \Hom_{\F\langle M\rangle}(X,Y)$, that is, the set of $Y\times X$ matrices with values in $M$ with at most one nonzero entry in every row and column. This is an $\F$-ring with composition defined by matrix multiplication. This is possible because all sums that occur in the product of two matrices with at most one non-zero entry in every row and column range over at most one term. This construction yields a faithful functor 
$$\F\langle-\rangle: \cM_0 \lto \{\F\text{-rings}\}$$ 
(cf.\ \cite[\S 2.3, Example 2]{Haran2007}), which has a right adjoint functor, namely, the functor $-_{[1],[1]}$ that takes any $\F$-ring $A$ to the monoid $A_{[1],[1]}=\Hom_A([1],[1])$. Composition with the functor $+0$ from section \ref{mtocc} yields a faithful functor from $\cM$ to the category of $\F$-rings, which admits a right adjoint given by composing $-_{[1],[1]}$ with the forgetful functor.

	Since ideals, localization and gluing in the category of $\F$-rings are defined in terms of the underlying monoid $A_{[1],[1]}$, which is in complete analogy to the construction of $\cM_0$-schemes, we obtain a faithful functor
$$\F\langle-\rangle: \{\cM_0\text{-schemes}\} \lto \{\F\text{-schemes}\}.$$ 
Composing with $+0$ yields a faithful functor from the category of $\cM$-schemes to the category of $\F$-schemes.


\subsection{$\cM_0$-schemes as generalized schemes with zero}
\label{mtodurov}

In this section we mention some relations between the category of $\cM$-schemes and the one of Durov's generalized schemes with zero. What follows is based on a work in progress by Peter Arndt (cf.\ \cite{ArndtUN}).

Associated to any monoid with zero $M$, we can construct the monad $T_M:\Sets \to \Sets$ that takes any set $X$ into $T_M(X):=(M\times X)/\!\!\sim$, where we identify all the elements of the form $(0,x)$ and assume that if $X= \varnothing$ the quotient consists of one element (the empty class). This monad is algebraic since cartesian products commute with filtered colimits, so we have a functor $T_{-}$ from $\cM_0$ to the category of generalized rings with $0$. The functor $T_{-}$ has as a right adjoint, namely, the functor $|-|$ that associates to a generalized ring $A$ its underlying monoid $|A| := A([1])$.

The functor $T_{-}$ commutes with localizations because localizations in generalized rings are defined in terms of localizations of the underlying monoids. Thus it naturally extends to an embedding of categories
$$ T_-:\{\cM_0\text{-schemes}\}\longrightarrow\{\text{Generalized schemes with }0\}. $$


\subsection{Relation between Durov's generalized rings with zero and Haran's $\F$-rings}
\label{htodurov}

In this section, we explain the relation between the categories of generalized rings with zero defined by Durov and the one of $\F$-rings defined by Haran, following some remarks by Durov (cf.\ \cite[\S 5.3.25]{Durov2007}) and a work in progress by Peter Arndt (cf.\ \cite{ArndtUN}).

Given a generalized ring with zero $T$, we can construct the $\F$-ring $T^{\cD}$ defined by the sets of morphisms
	\[
		T^{\cD}_{Y,X} = \Hom_{T^{\cA}}(X,Y):= \left\{T(f)|\ f\in \Hom_{\Sets}(X,Y)\right\}
	\]
obtained by applying $T$ to set maps between $X$ and $Y$. This construction yields a functor 
	\[ 
		\cD:\{\text{Generalized rings with }0\} \lto \{\F\text{-rings}\}, 
	\]
admitting a left adjoint $\cA$, where for every $\F$-ring $R$ the monad $\cA_R$ is defined by $\cA_R([n]) := R_{[1],[n]}$.

It is worth noting that the functor $\cD$ that sends generalized rings to $\F$-rings is not monoidal. This is due to the fact that the product $\circledcirc$ in the category of $\F$-rings is not compatible with the tensor product of generalized rings, cf.\ \cite[Remark 7.19]{Haran2007} and \cite{ArndtUN} for further details.

It seems to be likely that the pair of functors above lift to the corresponding categories of schemes, providing an adjunction
	\[
		\xymatrix{
			\{\text{Generalized schemes with }0\}\ar@<.5ex>[rr]^<<<<<<<<<<{\cD} && \{\F\text{-schemes}\} \ar@<.5ex>[ll]^<<<<<<<<<<{\cA}
		}.
	\]

One of the interesting applications of the functor $\cA$ is that we can define a base change functor from $\F$-rings to usual rings by composing $\cA$ with Durov's base change functor. This extension of scalars is left adjoint to the inclusion of rings into $\F$-rings. Complete details on this extension of scalars functor and its properties will be given in \cite{ArndtUN}.


\subsection{$\cM$-schemes as $\Lambda$-schemes}
\label{mtolambda}

In \cite{Borger2009}, Borger describes different examples of (flat) $\Lambda$-schemes. We will recall his constructions briefly to explain, why the category of $\cM$-schemes embeds into the category of $\Lambda$-schemes.

Given a monoid $A$, we can endow the (flat) scheme $X=\Spec\Z[A]$ with the following $\Lambda$-scheme structure $\Phi=\{\varphi_p: X\to X\}_{p\text{ prime}}$: for each prime $p$, the endomorphism $\varphi_p:X\to X$ is induced by the algebra map $\Z[A]\to\Z[A]$ given by $a\to a^p$ for every $a\in A$. Every morphism of monoids $A\to B$ induces therefore a morphism of the associated $\Lambda$-schemes $\Spec\Z[B]\to\Spec\Z[A]$.

Borger remarks in \cite[section 2.3]{Borger2009} that all small colimits and limits of $\Lambda$-schemes exist and that they commute with the base extension to $\Z$. Every $\cM$-scheme $X$ has an open cover by affine $\cM$-schemes $X_i$. By definition of the base extension to $\Z$ this yields an open cover of $X_\Z$ by $X_{i,\Z}$. Since the $X_{i,\Z}$ have a structure of a $\Lambda$-scheme and $X_\Z$ is the limit over the $X_{i,\Z}$ and their intersections, $X$ inherits the structure of a $\Lambda$-scheme. This yields the functor
$$ \cB: \ \{\cM\text{-schemes}\} \ \longrightarrow \  \{\Lambda\text{-schemes}\}. $$

Note, however, that $\cB$ is not essentially surjective as certain quotients of $\Lambda$-schemes are $\Lambda$-schemes that are not induced by an $\cM$-scheme. See \cite[sections 2.5 and 2.6]{Borger2009} for details.


\subsection{Toric varieties and affinely torified varieties}
\label{torictotorified}

As it was mentioned in section \ref{torifsch}, the orbit decomposition $\bdu_{\tau\in \Delta}T_\tau = \Spec \Z[A_\tau^\times]\lto X$ provides an affine torification of a toric variety $X$ with fan $\Delta$ (for details, cf.\ \cite[\S 1.3.3]{LL}). Moreover, every toric morphism between toric varieties induces a morphism between the corresponding affinely torified varieties.

In other words, we obtain an embedding of categories 
$$\iota: \{\text{Toric varieties}\} \lto \{\text{Affinely torified varieties}\}.$$


\subsection{Relation between affinely torified varieties, S-varieties and its variations}
\label{torifiedtos}

The relation between affinely torified varieties and S-varieties or CC-varieties is established in \cite[\S 2.2, \S 3.3]{LL}. Theorems 3.11 and 2.10 in loc.\ cit.\ provide embeddings $\mathcal{S}$ and $\mathcal{L}$ of the category of affinely torified varieties into the category of S-varieties resp.\ CC-varieties. 

In order to show that the above functors define varieties over $\Fun$ in the sense of Soul\'e or Connes-Consani, it is necessary to show that the universal property holds, which boils down to prove that a certain morphism $\vphi_\C:X \to V$ of affine complex varieties is actually defined over the integers, where $X$ is an affinely torified variety and $V$ is an arbitrary variety. The main idea of the proof is to consider for each irreducible component of $X$ the unique open torus $T_i$ in the torification of $X$ that is contained in the irreducible component (cf.\ \cite[Corollary 1.4]{LL}). Since a split torus satisfies the universal property of an S-variety resp.\ a CC-variety (cf.\ \cite{Soule2004} and \cite{Connes2008}), the morphism ${\vphi_\C}_{|T_i}$ is defined over $\Z$ and extends to a rational function $\psi$ on $X$ defined over the integers with the locus of poles contained in the complement of $T_i$. But the extension to $\C$ of the rational function is $\vphi_\C$, which has no poles. Thus $\psi$ cannot have a pole and is a morphism of schemes. 

Note that this proof that works for both S-varieties and CC-varieties can also be adopted to S${}^\ast$-varieties. This yields an embedding $\mathcal S^\ast$ of the category of affinely torified varieties into the category of S${}^\ast$-varieties.

In \cite{LL}, we find the construction of a partial functor $\mathcal{F}_{CC\to S}$ mapping CC-varieties to ``S-objects'', which, however, do not have to satisfy the universal property of an S-variety. In a similar fashion, replacing the complex algebra $\cA_X$ by the complexified scheme $X_\Z\otimes \C$, or conversely replacing the complex scheme $X_\C$ by the algebra of global sections, provide ways to compare the categories of CC-varieties and $\textrm{S}^\ast$-varieties. But there are technical differences between the two categories that prevent this correspondence to define a functor. Namely, the notions of gluing affine pieces are different and it is not clear how to define the grading for an CC-variety. Regarding the other direction, going from the complex scheme $X_\C$ to the algebra of global functions is a loss of information if the CC-variety is not affine. All these issues suggest that the reader should consider these categories as similar in spirit, but technically different.





\subsection{Generalized torified schemes and CC-schemes}
\label{torifiedtocc}

Torified schemes are closely connected to Connes and Consani's notion of $\Fun$-schemes. In this section, we generalize the notion of a torified scheme, which allows the sheaves to have multiplicative torsion and which provides an easier setting to compare the category of torified schemes with the category of CC-schemes.

A \dtext{generalized torified scheme} is a triple $\cX = (\widetilde{X}, X, e_X)$ where $\widetilde{X}$ is a geometric $\cM_0$-scheme, $X$ is a scheme, and $e_X:\widetilde{X}_\Z \to X$ is a morphism of schemes such that for every field $k$ the map $e_X(k):\widetilde{X}_\Z(k) \to X(k)$ is a bijection. A \dtext{morphism of torified schemes} is a pair of morphisms $\widetilde{\vphi}:\widetilde{X} \to \widetilde{Y}$ (morphism of $\cM_0$-schemes) and $\vphi:X\to Y$ (morphism of schemes) such that the diagram 
	\[
		\xymatrix{
			\widetilde{X}_\Z \ar[rr]^{\widetilde{\vphi}} \ar[d]_{e_X}&& \widetilde{Y}_\Z \ar[d]^{e_Y} \\
			X \ar[rr]^{\vphi}&& Y. 
		}
	\]
commutes. We denote the natural inclusion of the category of affinely torified varieties into the category of generalized torified schemes by $\iota$. We also have an obvious inclusion
	\[
		\cF':\{\textrm{$\cM_0$-schemes}\} \lto \{\textrm{generalized torified schemes}\}
	\]
constructed in the same fashion as the functor $\cF$ in section \ref{mtocc}.

Let $X$ be a scheme together with a torification $e_X:T=\bdu_i \Gm^{d_i} \lto X$. Then $T=\widetilde X_\Z$ for the $\cM_0$-scheme $\widetilde X= \bdu_i \mathbb{G}_{m,\cM_0}^{d_i}$. Using this description, the relation between (generalized) torified schemes and CC-schemes becomes apparent.

\begin{theorem}
	The functor
	\begin{eqnarray*}
		\mathcal{I} : \{\text{generalized torified schemes}\} & \lto & \{\text{CC-schemes}\} \\
		(\widetilde{X},X,e_X) & \lmto & (\uline{\widetilde{X}}(-), \uline{X}(-), \uline{e}_X)
	\end{eqnarray*}
	is an equivalence of categories.
\end{theorem}

\begin{proof}
	First of all note that the objects of both categories satisfy the condition that $\widetilde X_\Z(k)=X(k)$ for every fields $k$. The theorem follows from the facts that every $\cM_0$-scheme $\widetilde X(-)$ is represented by a geometric $\cM_0$-scheme (\cite[Prop.\ 3.16]{Connes2009}) and that a natural transformation of representable functors is induced by a morphism between the representing objects (Yoneda's lemma).
\end{proof}

This theorem establishes toric varieties, Grassmann and Schubert varieties and split reductive groups as CC-schemes. We have an immediate equivalence $\mathcal{I}\circ \cF' \simeq \cF$. 

\begin{remark}\label{rem_gen_tor}
	If $\cX =(\widetilde{X}, X, e_X)$ is a generalized torified scheme, then $\cX$ is torified if and only if there is an isomorphism $\widetilde{X}\cong \bdu_{x\in \widetilde{X}} \mathbb G_{m,\cM_0}^{\rk x}$ where $\rk x$ is the rank of the group $\cO_{\widetilde X,x}^\times$ and if the restriction $e_X{}_{\mid \mathbb G_{m}^{\rk x}}$ is an immersion for every $x\in \widetilde X$. As a consequence we obtain the following result.
\end{remark}

\begin{corollary}
	Let $\cX = (\widetilde{X}, X, e_X)$ be a CC-scheme. If $\widetilde{X}$ is integral and torsion-free and if $e_X{}_{\mid \widetilde Y_\Z}$ is an immersion for every connected component $\widetilde Y$ of $\widetilde X$, then $X=\cX_\Z$ is torified.
\end{corollary}

\begin{remark}
\label{non-affine_torification}
The torification of $\Gr(2,4)$ given by a Schubert decomposition is not affine (cf.\ \cite[Section 1.3.4]{LL}), what makes it unlikely that the corresponding CC-scheme has an open cover by open sub-CC-schemes of the form $(\widetilde X,X,e_X)$ such that both $\widetilde X$ and $X$ are affine. This is a reason for the flexibility of CC-schemes, and it seems that the category of CC-schemes is different to categories that are obtained by gluing affine pieces. In particular, cf.\ Remark \ref{rem_salch}.
\end{remark}


\subsection{The map of $\Fun$-land}
\label{map}

The functors that we described in the previous section are illustrated in the commutative diagram in Figure \ref{bigpicture} of categories and functors. Note that we place the category of schemes on the outside in order to have a better overview of the different categories. In Figure \ref{bigpicture}, all the extensions of scalars from $\Fun$ to $\Z$ are denoted by $-\ot_{\Fun} \Z$, the functor $\iota$ is the canonical inclusion of (affinely) torified schemes into generalized torified schemes. All the other functors have been defined in the previous sections. 
As we are dealing with categories and functors, the diagram is commutative only up to isomorphism. Some of the functors admit adjoints also described in the text, that in many cases have been left out of the diagram for the sake of clarity. The dashed arrows corresponding to the functors $\cD$ and $\cA'$ represent a work in progress. 
	
A few remarks about the commutativity (up to isomorphism) of the diagram. Any path that starts at toric varieties and ends up at schemes will always produce the toric variety itself, so all these possible paths are equivalent. Commutativity of subdiagrams involving S-varieties, CC-varieties was proven in \cite{LL}, and extends verbatim to $\textrm{S}^\ast$-varieties. The commutativity of the triangle involving CC-schemes, generalized schemes with 0 and H-schemes (also involving the functor $\cA$ adjoint of $\cD$ is explained in Arndt's work \cite{ArndtUN}. The equivalence $(-\ot_{\Fun}\Z) \circ \mathcal{B}\simeq (-\ot_{\Fun}\Z) \circ \mathcal{M}_1$ follows immediately from the definitions.

   \begin{figure}[hbtp]
    \begin{center}
    \includegraphics[width=1.0\textheight,angle=90]{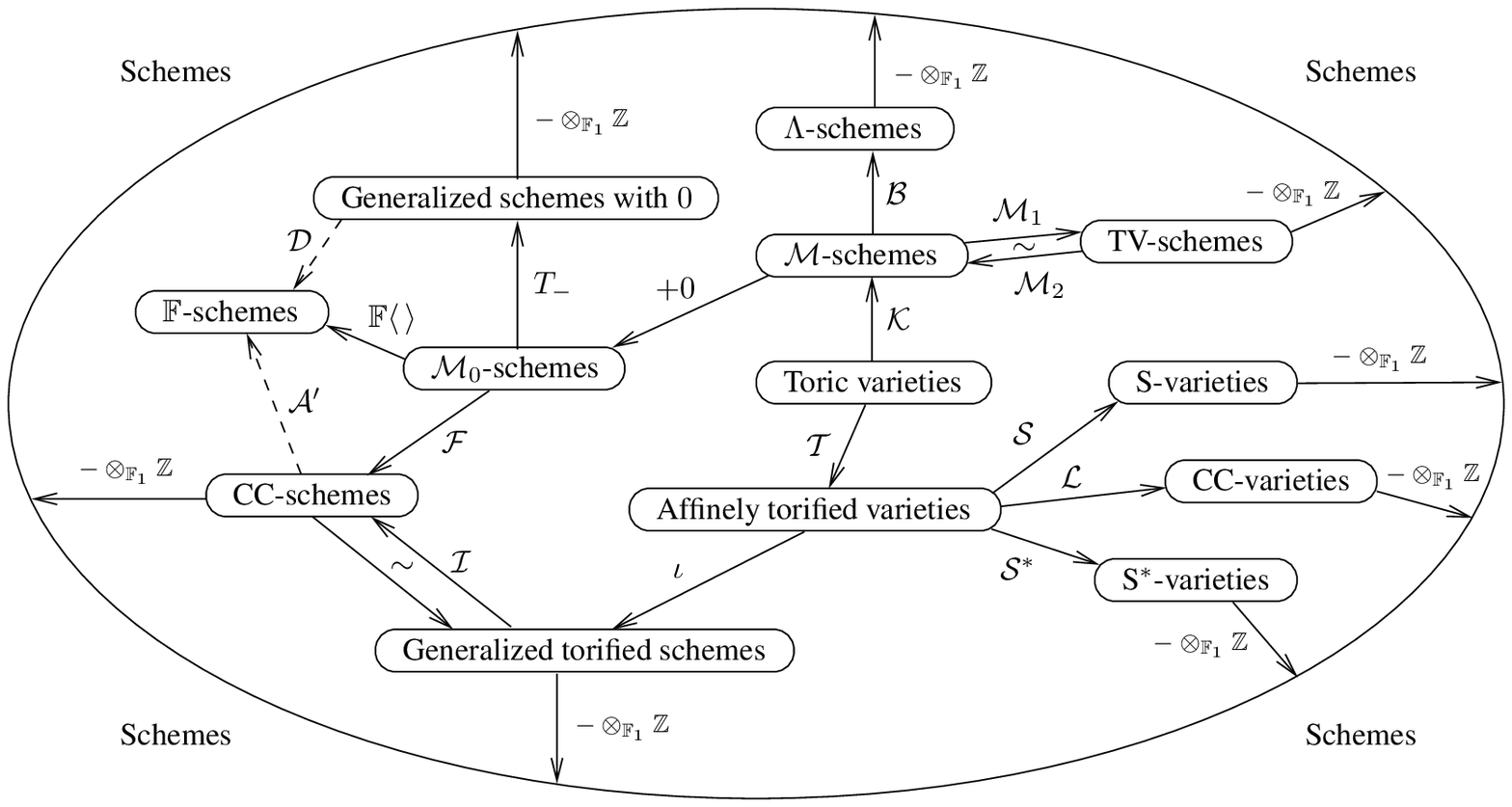}
     \caption{The map of $\Fun$-land} \label{bigpicture}
    \end{center} 
   \end{figure}


\subsection{Algebraic groups over $\Fun$} 
\label{algebraic_groups}

The aim of Connes and Consani's paper \cite{Connes2008} as described in section \ref{ccvarieties} was to realize Tits' idea from \cite{Tits1957} of giving the Weyl group of a split reductive group scheme (cf.\ \cite[Expos\'e XIX, Def.\ 2.7]{SGA3}) an interpretation as a ``Chevalley group over $\Fun$''. Connes and Consani obtained partial results, namely, that every split reductive group scheme $G$ has a model ``over $\F_{1^2}$'' and that the normalizer $N$ of a maximal split torus $T$ of $G$ can be defined as a group object over $\F_{1^2}$. However, there is no model of $G$ as an algebraic group over $\Fun$ or $\F_{1^2}$. For a further discussion of possibilities and limitation in this context, see \cite[section 6.1]{LL}.

In the category of CC-schemes, the same reasons prevent split reductive group schemes (except for tori) to have a model as an algebraic group over $\Fun$. However, the second author of this paper showed that this category is flexible enough to invent new notions of morphisms that yield the desired results. We review the definitions and the main result  from \cite{Lorscheid2009}.

Let $\widetilde X$ be an $\cM_0$-scheme. Recall the definition of the rank $\rk x$ of $x\in\widetilde X$ from Remark \ref{rem_gen_tor}. We define the sub-$\cM_0$-scheme $\widetilde X^\rk\hookrightarrow \widetilde X$ as the disjoint union $\coprod \Spec_{\cM_0}\cO_{\widetilde X,x}^\times$ over all points $x\in\widetilde X$ of minimal rank. A \emph{strong morphism} between $\Fun$-schemes $(\widetilde Y,Y,e_Y)$ and $(\widetilde X,X,e_X)$ is a pair $(\widetilde\varphi,\varphi)$, where $\widetilde\varphi:\widetilde Y^\rk\to\widetilde X^\rk$ is a morphism of $\cM_0$-schemes and $\varphi:Y\to X$ is a morphism of schemes such that the diagram  
$$\xymatrix{\widetilde Y^\rk_\Z\ar[rr]^{\widetilde f_\Z}\ar[d]_{e_Y} && \widetilde X^\rk_\Z\ar[d]^{e_X}\\ Y\ar[rr]^f && X} $$
commutes.

The morphism $\spec\O^\times_{X,x}\to\ast_{\cM_0}$ to the terminal object $\ast_{\cM_0}=\Spec_{\cM_0}\{0,1\}$ in the category of $\cM_0$-schemes induces a morphism 
$$ t_{\widetilde X}: \quad \widetilde X^\rk \ = \coprod_{x\in\widetilde X^\rk}\spec\O^\times_{X,x} \quad \longrightarrow \quad \ast_{\widetilde X} \ := \coprod_{x\in\widetilde X^\rk}\ast_{\cM_0}. $$
Given a morphism $\widetilde \varphi:\widetilde Y^\rk\to\widetilde X^\rk$ of $\cM_0$-schemes, there is thus a unique morphism $t_{\widetilde \varphi}:\ast_{\widetilde Y}\to\ast_{\widetilde X}$ such that $t_{\widetilde\varphi}\circ t_{\widetilde Y}=t_{\widetilde X}\circ\widetilde\varphi$. Let $X^\rk$ denote the image of $e_X: \widetilde X^\rk_\Z\to X$. A \emph{weak morphism} between $\Fun$-schemes $(\widetilde Y,Y,e_Y)$ and $(\widetilde X,X,e_X)$ is a pair $(\widetilde\varphi,\varphi)$, where $\widetilde\varphi:\widetilde Y^\rk\to\widetilde X^\rk$ is a morphism of $\cM_0$-schemes and $\varphi:Y\to X$ is a morphism of schemes such that the diagram  
$$\xymatrix@R=1pc{\widetilde Y^\rk_\Z\ar[rr]^{\widetilde \varphi_\Z}\ar[dr]_{t_{\widetilde Y,\Z}} && \widetilde X^\rk_\Z\ar[dr]^{t_{\widetilde X,\Z}}\\
                 &(\ast_{\widetilde Y})_\Z\ar[rr]^{t_{\widetilde\varphi,\Z}} && (\ast_{\widetilde X})_\Z \\  Y^\rk\ar[rr]^\varphi\ar[ur] && X^\rk\ar[ur]} $$
commutes. 

If $\cX=(\widetilde X,X,e_X)$ is an CC-scheme, then we define \emph{the set of $\Fun$-points $\cX(\Fun)$} as the set of strong morphism from $(\ast_{\cM_0},\Spec\Z,\id_{\Spec\Z})$ to $(\widetilde X,X,e_X)$. An \emph{algebraic group over $\Fun$} is a group object in the category whose objects are $\Fun$-schemes and whose morphisms are weak morphisms. The base extension functor $-\otimes_\Fun\Z$ from this category to the category of schemes is given by sending $\cX=(\widetilde X,X,e_X)$ to $\cX_\Z=X$. On the other hand, the group law of an algebraic group $\cG$ over $\Fun$ induces a group structure on the set $\cG(\Fun)$. We realize Tits' idea in the following form.

\begin{theorem}[{\cite[Thm.\ 7.9]{Lorscheid2009}}] \label{split_reductive_over_fun}
 For every split reductive group scheme $G$ with Weyl group $W$, there exists an algebraic group $\cG$ over $\Fun$ such that $\cG_\Z\simeq G$ as group schemes and $\cG(\Fun)\simeq W$ as groups.
\end{theorem}


\end{document}